\documentclass[12pt]{article}
\usepackage{amsmath,graphicx,amssymb,amsthm}
\usepackage[top=2.5cm,bottom=3cm,left=3cm,right=3cm]{geometry}
\usepackage{hyperref,color,mathrsfs}
\allowdisplaybreaks
\newtheorem{lemma}{Lemma}
\newtheorem{corollary}[lemma]{Corollary}

\newtheorem{proposition}[lemma]{Proposition}

\newtheorem{theorem}[lemma]{Theorem}
\newtheorem{conjecture}[lemma]{Conjecture}

\begin{document}

\title{Hausdorff operators on weighted mixed norm Fock spaces}

\author{Yongqing Liu\footnote{Corresponding author. \newline \indent \;\, E-mail: yongqingliu@cslg.edu.cn (Y. Liu).}\\
\footnotesize{School of Mathematics and Statistics, Changshu Institute of Technology,}\\
\footnotesize{Changshu, Jiangsu 215500, China}}

\date{}
\maketitle

\noindent\textbf{Abstract} In this paper, we study Hausdorff operator $\mathcal{H}_\mu$ on a large class of weighted mixed norm Fock spaces $F_\phi^{p,q}$ for $1\leq p,q\leq\infty$. The boundedness and compactness of $\mathcal{H}_\mu$ on $F_\phi^{p,q}$ are characterized. As applications, we give when Hausdorff operator on $F_\phi^{p,q}$ is power bounded or uniformly mean ergodic.

\vspace{5pt}

\vspace{5pt}

\section{Introduction}


The investigation of Hausdorff operators has a long history. The papers \cite{GS} and \cite{Liflyand} provide a nice historical summary and survey on this topic, to which we refer. Let $\mathbb{C}$ be the complex plane and denote by $H(\mathbb{C})$ the space of entire functions. Following \cite{GS}, given a positive Borel measure $\mu$ on $(0,\infty)$, we consider formally the Hausdorff operator induced by the measure $\mu$ defined by
\begin{equation*}
\mathcal{H}_\mu f(z)=\int_{(0,\infty)} \frac{1}{t} f  \bigg(\frac{z}{t}\bigg) d\mu(t),   \quad   z\in \mathbb{C},    
\end{equation*}
where $f\in H(\mathbb{C})$. There are many classical operators in analysis which are special cases of the Hausdorff operator if one chooses suitable measure $\mu$, for instance, Flett's fractional integral \cite[p. 748]{Flett}. Indeed, let $\beta\in\mathbb{R}$ and $f(z)=\sum_{n=0}^\infty a_nz^n$ in $H(\mathbb{C})$. The fractional integral $\mathcal{I}^\beta$ of order $\beta$ is defined as
\begin{equation*}
\mathcal{I}^\beta f(z)=\sum_{n=0}^\infty (1+n)^{-\beta} a_nz^n.
\end{equation*}
For $\gamma>0$, consider the measure
\begin{equation*}
d\mu_\gamma(t)=\frac{\log^{\gamma-1}t}{t}\frac{dt}{\Gamma(\gamma)}
\end{equation*}
on $[1,\infty)$. Then, for $n\in\mathbb{N}_0$,
\begin{equation*}
\int_1^\infty \frac{1}{t^{n+1}}d\mu_\gamma(t)=(1+n)^{-\gamma},
\end{equation*}
and hence $\mathcal{H}_{\mu_\gamma} = \mathcal{I}^\gamma$. In particular, when $\gamma=1$, one obtains the Hardy operator
\begin{equation*}
Hf(z)=\frac{1}{z}\int_0^z f(w)dw.
\end{equation*}

Concrete operators on weighted mixed norm space $H_\omega^{p,q}$ on the unit disk, for example, Bergman projections, Carleson embeddings, and fractional derivatives, have been intensively studied for a long time (see \cite[Chapter 7]{JVA}, \cite[Chapter 3]{Pavlovic}, \cite{PRS, ZFGH} and the references therein). However, very little seems to be known about the behavior of operator on weighted mixed norm spaces of entire functions. Recently, in \cite{Liu}, we studied Fock projections on mixed norm spaces on $\mathbb{C}$. In this paper, we are mainly concerned with Hausdorff operator $\mathcal{H}_\mu$ on weighted mixed norm Fock spaces $F_\phi^{p,q}$.

Let $\phi:[0,\infty)\to \mathbb{R}^+$ be a continuously differentiable function. It is radial if $\phi(z)=\phi(|z|)$ for all $z\in\mathbb{C}$. We call $\phi$ a weight if the function $\phi$ is radial and $0<q<\infty$ such that $\int_0^\infty e^{-q\phi(r)}rdr<\infty$. For $0\leq r<\infty$ and $f\in H(\mathbb{C})$, set
\begin{equation*}
M_p(f,r)=\bigg(\frac{1}{2\pi}\int_0^{2\pi}|f(re^{i\theta})|^pd\theta\bigg)^\frac{1}{p}, \quad 0<p<\infty,
\end{equation*}
and
\begin{equation*}
M_\infty(f,r)=\sup_{|z|=r}|f(re^{i\theta})|.
\end{equation*}
It is well-known that $M_p(f,r)$ is increasing with $r\in [0,\infty)$. For $0<p\leq\infty$ and $0<q<\infty$, the weighted mixed norm Fock space $F_\phi^{p,q}$ consists of $f\in H(\mathbb{C})$ such that
\begin{equation*}
\|f\|_{p,q,\phi}^q=\int_0^\infty M_p^q(f,r)e^{-q\phi(r)}rdr.
\end{equation*}
For $0<p\leq\infty$, the weighted mixed norm Fock space $F_\phi^{p,\infty}$ consists of $f\in H(\mathbb{C})$ such that
\begin{equation*}
\|f\|_{p,\infty,\phi}=\sup_{r\in[0,\infty)} M_p(f,r)e^{-\phi(r)}.
\end{equation*}
Then $F_\phi^{p,q}$ is a Banach space for $1\leq p,q\leq\infty$, else $F_\phi^{p,q}$ is an $s$-Banach space with $s=\min\{p,q\}$ (see \cite{CHL}). If $p=q$, then $F_\phi^{p,q}$ is just the weighted Fock space $F_\phi^p$ (see \cite{BBB, CHLS, HHLS, SY}). As usual, $F_\alpha^p$ denotes the classical Fock space induced by the standard Gaussian weight $\phi(z)=\frac{\alpha}{2}|z|^2, \alpha>0$ (see \cite{Zhu2} for the theory of Fock spaces); $F^{p,q}_\alpha$ denotes the classical mixed norm Fock spaces (see \cite[Section 2.1]{FT}). Moreover, by \cite[Theorem 2.1]{Lusky}, the polynomials are dense whenever $1\leq p\leq \infty$ and $1\leq q<\infty$. In particular, $F_\phi^{p,q}$ is separable. 
The weighted mixed norm Fock spaces $F_\phi^{p,q}$ arise naturally in operator theory and (random) analytic function spaces (see \cite{BG, FT, HCL, Liu}).

Let us briefly review the relevant works on this topic. First of all, the boundedness and compactness of $\mathcal{H}_\mu$ on $F_\alpha^p$ were characterized in \cite{GS}, furthermore, the authors \cite[Section 3.7]{BFR} showed that a bounded Hausdorff operator on $F_\alpha^2$ automatically defines an element in the Toeplitz algebra $\mathcal{T}$. In \cite[Section 2]{Bonet}, Bonet obtained the boundedness and compactness results for $\mathcal{H}_\mu$ on weighted Banach spaces of type $H^\infty$ (similar to $F_\phi^\infty$). Recently, Blasco and Galbis \cite{BG} characterized the boundedness and compactness of Hausdorff operators on mixed norm Fock spaces $F_\alpha^{p,q}$, in particular, on weighted spaces of type $H^\infty$ (see \cite[Section 7]{BG}).

Motivated by the works in \cite{Bonet} and \cite{BG}, we are interested in the case that Hauesdorff operator is bounded or compact on weighted mixed norm Fock spaces $F_\phi^{p,q}$. On the other hand, as an important example, the Hardy operator $H$ on different weighted Banach spaces of entire functions has been proven to be power bounded, uniformly mean ergodic, $\|H\|=1$, and $H^2$ is compact (see \cite[Section 3.3]{BBF}, \cite[Section 4]{Beltran}, \cite[Section 4]{BMW}). Inspired by these researches of Hardy operator, it is therefore of our second interest to answer the question of when Hausdorff operators are power bounded or uniformly mean ergodic on $F_\phi^{p,q}$.

We briefly discuss the classes of weights for the main results of this paper here. The class of weights $\mathcal{W}_p$ was introduced in \cite[Definition 1.1]{HCL} (see also \cite[Lemma 17]{CP}) to study a Littlewood-Paley formula for the weighted mixed norm Fock spaces. Assume that $\phi:[0,\infty)\to \mathbb{R}^+$ is twice continuously differentiable and there exists $r_0>0$ such that $\phi'(r)\neq0$ for $r>r_0$.
We recall that $\phi$ is in the class $\mathcal{W}_p$, $0<p<\infty$, if it satisfies the following conditions:
\begin{equation*}
\lim_{r\to\infty}\frac{re^{-p\phi(r)}}{\phi'(r)}=0,
\end{equation*}
\begin{equation*}
\limsup_{r\to\infty}\frac{1}{r}\bigg(\frac{r}{\phi'(r)}\bigg)'<p,
\end{equation*}
\begin{equation*}
\liminf_{r\to\infty}\frac{1}{r}\bigg(\frac{r}{\phi'(r)}\bigg)'>-\infty.
\end{equation*}

It is worth to point out that the class $\mathcal{W}_p$ is quite large, including functions whose growth ranges from logarithmic (e.g., $\phi(r)=a\ln(1+r)$, $ap>2$) to highly exponential (e.g., $\phi(r)=e^{e^r}$). In addition, we find that a weight $\phi$ satisfies
\begin{equation*}
\lim_{r\to\infty}\frac{1}{r}\bigg(\frac{r}{\phi'(r)}\bigg)'=0
\end{equation*}
if and only if
\begin{equation*}
\lim_{r\to\infty} r\phi'(r)=\infty,  \quad  \lim_{r\to\infty}\frac{\phi''(r)}{(\phi'(r))^2}=0.
\end{equation*}
In fact, a simple calculation shows that
\begin{equation*}
\frac{1}{r}\bigg(\frac{r}{\phi'(r)}\bigg)'=\frac{1}{r\phi'(r)}-\frac{\phi''(r)}{(\phi'(r))^2}.
\end{equation*}
This, together with \cite[Lemma 2.3]{CHL}, gives the conclusion.

The method of proof for the boundedness of Hausdorff operators is available in all cases, including the class $\mathcal{W}_p$. However, the method we use to prove the compactness of Hausdorff operator depends heavily on weights, we then restrict our weight $\phi$ to the class $\mathcal{C}$, which is a rather weak condition that is satisfied. Let $\phi:[0,\infty)\to \mathbb{R}^+$ is continuously differentiable. We say that a weight $\phi$ belongs to the class $\mathcal{C}$ if
\begin{equation*}
\lim_{r\to\infty} r\phi'(r)=\infty.
\end{equation*}
This class $\mathcal{C}$ includes the power functions $\phi(r)=r^l$ with $l>0$, $\phi(r)=(\ln r)^\alpha$ for $\alpha>1$, and exponential type functions, such as $\phi(r)=e^{\beta r}$ for $\beta>0$.

This paper is organized as follows. In Section \ref{section3}, we provide a characterization of when Hausdorff operator $\mathcal{H}_\mu$ is bounded on weighted mixed norm Fock spaces $F_\phi^{p,q}$ (Theorem \ref{boundedness}).
We also completely characterize the compactness of $\mathcal{H}_\mu$ on $F_\phi^{p,q}$, under the assumption of $\phi\in \mathcal{C}$ (Theorem \ref{compactness}).
As some applications, in Section \ref{section4}, we prove that Hauesdorff operator $\mathcal{H}_\mu$ is power bounded on $F_\phi^{p,q}$ if and only if it is a contraction (Theorem \ref{power bounded}). Then, we describe the uniformly mean ergodicity of Hauesdorff operator $\mathcal{H}_\mu$ on $F_\phi^{p,q}$ according to its boundedness and compactness, respectively (Theorem \ref{compact UME} and Theorem \ref{bounded UME}) .
Finally, in Section \ref{section5}, we discuss some problems related to this article.

Throughout this paper, we use $C$ to denote a positive number, which may vary from place to place. For two quantities $A$ and $B$, $A\lesssim B$ means that there exists a constant $C>0$, independent of the involved variables, such that $A\leq CB$, and we write $A\simeq B$ when $A\lesssim B$ and $B\lesssim A$. In what follows, $\mathbb{N}$ stands for the natural numbers and we set $\mathbb{N}_0=\mathbb{N}\cup\{0\}$.

\section{Boundedness and compactness of $\mathcal{H}_\mu$}\label{section3}

This section is devoted to characterize the boundedness and compactness of Hausdorff operators. Firstly, for $t>0$, the dilation operator $\mathcal{D}_t$ is defined by
\begin{equation*}
\mathcal{D}_tf(z)=f\bigg(\frac{z}{t}\bigg),  \quad  z\in\mathbb{C}.
\end{equation*}
Then Hausdorff operator $\mathcal{H}_\mu$ commutes with $\mathcal{D}_t$, that is, $\mathcal{D}_t \mathcal{H}_\mu=\mathcal{H}_\mu \mathcal{D}_t$. Moreover, the following proposition indicates that dilation operator $\mathcal{D}_t$ is bounded on weighted mixed norm Fock spaces.
\begin{proposition}\label{proposition}
Let $1\leq p,q \leq\infty$ and $\phi$ be a weight. The dilation operator $\mathcal{D}_t$ is bounded on $F_\phi^{p,q}$ if and only if $t\geq1$. Moreover, we have
\begin{equation*}
\|\mathcal{D}_t\|=1,  \quad  \forall t\geq 1.
\end{equation*}
\end{proposition}
\begin{proof}
For $t\geq1$, we have
\begin{equation*}
M_p(\mathcal{D}_tf,r)\leq M_p\bigg(f,\frac{r}{t}\bigg) \leq M_p(f,r).
\end{equation*}
From this, $\mathcal{D}_t$ is bounded on $F_\phi^{p,q}$ and $\|\mathcal{D}_t\|\leq1$. In particular, take $f(z)=1$, we obtain $\|\mathcal{D}_t\|\geq1$. This completes the proof.
\end{proof}

We now state the first main result in this section.
\begin{theorem}\label{boundedness}
Let $1\leq p,q \leq\infty$, $\phi$ be a weight, and $\mu$ be a positive Borel measure on $(0,\infty)$. Then the Hausdorff operator $\mathcal{H}_\mu$ is bounded on $F_\phi^{p,q}$ if and only if
\begin{equation*}
\mu(0,1)=0  \quad  \text{and}  \quad  \int_{(0,\infty)} \frac{1}{t}d\mu(t)<\infty.
\end{equation*}
In this case, its norm is given by
\begin{equation}\label{equation2}
\|\mathcal{H}_\mu\|=\int_{[1,\infty)}\frac{1}{t}d\mu(t).
\end{equation}
\end{theorem}

For the proof of this theorem, we require the following lemma (see \cite[Theorem 4.3]{BG}, \cite[p. 3028]{GS}). Here, we include a sketch of its proof for the sake of completeness. For $n\in\mathbb{N}_0$, we denote $\mu_n=\int_{(0,\infty)}t^{-(1+n)}d\mu(t)$.
\begin{lemma}\label{lemma2}
Let $\mu$ be a positive Borel measure on $(0,\infty)$. Then
\begin{equation*}
\sup_{n\in\mathbb{N}_0}\mu_n=\sup_{n\in\mathbb{N}_0}\int_{(0,\infty)}t^{-(1+n)}d\mu(t)<\infty
\end{equation*}
if and only if
\begin{equation*}
\mu(0,1)=0 \quad  \text{and}  \quad  \int_{[1,\infty)} \frac{1}{t}d\mu(t)<\infty.
\end{equation*}
\end{lemma}
\begin{proof}
If $\mu(0,1)=0$, then
\begin{equation*}
\mu_n=\int_{[1,\infty)} \frac{1}{t^{n+1}}d\mu(t) \leq \int_{[1,\infty)} \frac{1}{t}d\mu(t)<\infty.
\end{equation*}
Therefore, we get $\sup_{n\in\mathbb{N}_0}\mu_n<\infty$.

Conversely, set $\varphi(t)=\frac{1}{t}$. Since $d\nu(t)=\frac{1}{t}d\mu(t)$ is a finite measure on $(0,\infty)$ and $\sqrt[n]{\mu_n}=\|\varphi\|_{L^n(d\nu)}$, the sequence $\{\sqrt[n]{\mu_n}\}_{n=0}^\infty$ is increasing. Therefore, the condition
\begin{equation*}
\sup_{n\in\mathbb{N}_0}\mu_n<\infty
\end{equation*}
gives that
\begin{equation*}
\sup_{n\in\mathbb{N}_0}\sqrt[n]{\mu_n}=\lim_{n\to\infty}\sqrt[n]{\mu_n} \leq 1.
\end{equation*}
By \cite[Lemma 4.1]{BG}, we get $\mu(0,1)=0$.
\end{proof}

Now, we come to the proof of Theorem \ref{boundedness}.
\begin{proof}[Proof of Theorem \ref{boundedness}]
Since $\mu(0,1)=0$, we have
\begin{equation*}
\mathcal{H}_\mu f(z)=\int_1^\infty \frac{1}{t} f \bigg(\frac{z}{t}\bigg) d\mu(t).
\end{equation*}
For $1\leq p\leq\infty$, by Minkowski's integral inequality, we get
\begin{equation*}
M_p(\mathcal{H}_\mu f,r) \leq \int_1^\infty  \frac{1}{t} M_p\bigg(f,\frac{r}{t}\bigg) d\mu(t) \leq \bigg(\int_1^\infty\frac{1}{t}d\mu(t)\bigg) M_p(f,r).
\end{equation*}
Multiplying both sides by $re^{-q\phi(r)}$ and integrating over $r$ from $0$ to $\infty$ yields that $\mathcal{H}_\mu$ is bounded on $F_\phi^{p,q}$, and we have
\begin{equation}\label{equation1}
\|\mathcal{H}_\mu\| \leq \int_{[1,\infty)}\frac{1}{t}d\mu(t).
\end{equation}

Conversely, if the operator $\mathcal{H}_\mu$ is bounded on $F_\phi^{p,q}$, take $h_n(z)=z^n$, $n\in\mathbb{N}_0$, then
\begin{equation*}
\mathcal{H}_\mu h_n(z)=\int_0^\infty \frac{1}{t^{n+1}} d\mu(t) h_n.
\end{equation*}
Moreover, for each $n\in\mathbb{N}_0$,
\begin{equation*}
\int_0^\infty \frac{1}{t^{n+1}} d\mu(t)\|h_n\|_{p,q,\phi} \leq \|\mathcal{H}_\mu\| \|h_n\|_{p,q,\phi}.
\end{equation*}
This implies
\begin{equation*}
\sup_{n\in\mathbb{N}_0}\int_0^\infty \frac{1}{t^{n+1}} d\mu(t) \leq \|\mathcal{H}_\mu\|.
\end{equation*}
Then Lemma \ref{lemma2} implies that $\mu(0,1)=0$. On the one hand, we have that
\begin{equation*}
\int_{[1,\infty)}\frac{1}{t}d\mu(t)=\sup_{n\in\mathbb{N}_0}\int_{[1,\infty)} \frac{1}{t^{n+1}}d\mu(t)<\infty.
\end{equation*}
This, together with (\ref{equation1}), gives equality (\ref{equation2}) of the norm and the integral. The proof is complete.
\end{proof}

The following conclusion indicates that $\mathcal{H}_\mu$ acts as a certain coefficient multiplier on $F_\phi^{p,q}$.
\begin{corollary}
Let $\mu$ be a positive Borel measure on $(0,\infty)$. If
\begin{equation*}
 \sup_{n\in\mathbb{N}_0}\mu_n=\sup_{n\in\mathbb{N}_0}\int_{(0,\infty)} \frac{1}{t^{n+1}}d\mu(t)<\infty,
\end{equation*}
then for every entire function
\begin{equation*}
f(z)=\sum_{n=0}^\infty a_n z^n \in F_\phi^{p,q},
\end{equation*}
we have
\begin{equation*}
\mathcal{H}_\mu f(z)=\sum_{n=0}^\infty \mu_n a_n z^n \in F_\phi^{p,q}.
\end{equation*}
\end{corollary}



The second theorem of this section deals with the problem of compactness of $\mathcal{H}_\mu$.
\begin{theorem}\label{compactness}
Let $1\leq p,q\leq\infty$, $\phi\in\mathcal{C}$, and $\lim\limits_{n\to\infty} \|z^n\|_{p,q,\phi}^{1/n}=\infty$. Let $\mu$ be a positive Borel measure on $(0,\infty)$. Then the Hausdorff operator $\mathcal{H}_\mu$ is compact on $F_\phi^{p,q}$ if and only if
\begin{equation*}
\mu((0,1])=0 \quad and \quad \int_{(0,\infty)}\frac{1}{t}d\mu(t)<\infty.
\end{equation*}
\end{theorem}

Our proof of Theorem \ref{compactness} requires the following two lemmas. The first lemma plays a key ingredient in the study of compactness
of linear operators.

\begin{lemma}[\cite{ACFPP} Lemma 4.10]\label{lemma3}
Let $X,Y$ be two Banach spaces of analytic functions on $\mathbb{C}$. Suppose that the following conditions are satisfied:

$\mathrm{(a)}$ the point evaluation functionals on $Y$ are continuous;

$\mathrm{(b)}$ the closed unit ball of $X$ is a compact subset of $X$ in the topology of uniform convergence on compact sets;

$\mathrm{(c)}$ $T:X \to Y$ is continuous when $X$ and $Y$ are given the topology of uniform convergence on compact sets.

Then $T$ is a compact operator if and only if given a bounded sequence $\{f_n\}$ in $X$ such that $f_n\to0$ uniformly on compact sets, then the sequence $\{Tf_n\}$ converges to zero in the norm of $Y$.
\end{lemma}

It is worth mentioning that conditions $\mathrm{(a)}$ and $\mathrm{(b)}$ of the previous lemma hold when $X=Y=F_\phi^{p,q}$, and in case Hausdorff operator satisfies $\mathrm{(c)}$. The second lemma provides a characterization that the mass of the measure $\mu$ at 1 is zero.

\begin{lemma}\label{lemma4}
Let $\mu$ be a positive Borel measure on $(0,\infty)$ such that $\mu(0,1)=0$ and $\int_{[1,\infty)}t^{-1}d\mu(t)<\infty$. Then
\begin{equation*}
\lim_{n\to\infty}\int_{[1,\infty)} \frac{1}{t^{n+1}}d\mu(t)=0
\end{equation*}
if and only if $\mu(\{1\})=0$.
\end{lemma}
\begin{proof}
First from the relation
\begin{equation*}
\int_{[1,\infty)} \frac{1}{t^{n+1}}d\mu(t)=\mu(\{1\})\cdot1 + \int_{(1,\infty)} \frac{1}{t^{n+1}}d\mu(t),
\end{equation*}
by dominated convergence theorem with dominating function $\frac{1}{t}$, we get
\begin{equation*}
\lim_{n\to\infty}\int_{(1,\infty)} \frac{1}{t^{n+1}}d\mu(t)=0,
\end{equation*}
showing that
\begin{equation*}
\lim_{n\to\infty}\int_{[1,\infty)} \frac{1}{t^{n+1}}d\mu(t)=\mu(\{1\}).
\end{equation*}
This yields the desired result.
\end{proof}

Now, we come to the proof of Theorem \ref{compactness}.
\begin{proof}[Proof of Theorem \ref{compactness}]
Assume that $\mu((0,1])=0$ and $\int_{(0,\infty)}t^{-1}d\mu(t)<\infty$. Then
\begin{equation*}
\mathcal{H}_\mu f(z)=\int_{[1,\infty)} \frac{1}{t} f \bigg(\frac{z}{t}\bigg) d\mu(t)
\end{equation*}
and it is bounded on $F_\phi^{p,q}$. For any $\delta>1$, consider
\begin{equation*}
\mathcal{H}_\mu^\delta f(z)=\int_{(\delta,\infty)} \frac{1}{t} f \bigg(\frac{z}{t}\bigg) d\mu(t).
\end{equation*}
We prove that $\mathcal{H}_\mu^\delta$ is compact on $F_\phi^{p,q}$. For this, let
\begin{equation*}
\varphi(z)=\phi\bigg(\frac{z}{\delta}\bigg).
\end{equation*}
It sufficiently proves that the operator $\mathcal{H}_\mu^\delta:F_\phi^{p,q}\to F_\varphi^{p,q}$ is bounded and the inclusion $\mathcal{J}:F_\varphi^{p,q}\to F_\phi^{p,q}$ is compact.

We first prove that the inclusion map $\mathcal{J}: F_\varphi^{p,q} \to F_\phi^{p,q}$ is compact. By Lemma \ref{lemma3}, it is enough to prove that if $\{f_n\}$ is a bounded sequence in $F_\varphi^{p,q}$ that converges to $0$ uniformly on compact subsets of $\mathbb{C}$ then $\lim\limits_{n\to\infty}\|f_n\|_{p,q,\phi}=0$. By Lagrange's theorem, there exists $r_0\in(\frac{r}{\delta},r)$ such that
\begin{equation*}
\phi(r)-\phi\bigg(\frac{r}{\delta}\bigg)=(1-\delta^{-1}) \phi'(r_0)r.
\end{equation*}
This, together with $\phi\in \mathcal{C}$, gives that
\begin{equation*}
\lim_{|z|\to\infty}e^{[\varphi(z)-\phi(z)]}=0.
\end{equation*}
So given $\varepsilon>0$, we fix $R_1\in(0,\infty)$ with
\begin{equation*}
e^{[\varphi(z)-\phi(z)]}<\varepsilon, \quad |z|>R_1.
\end{equation*}
We shall consider the case $1\leq q<\infty$, the proof for $q=\infty$ is similar. Let $1\leq p\leq\infty$ and $1\leq q<\infty$. Then we get
\begin{align*}
\|f_n\|_{p,q,\phi}^q
&=\int_0^\infty M_p^q(f_n,r) e^{-q\phi(r)}rdr\\
&=\int_0^{R_1} M_p^q(f_n,r) e^{-q\phi(r)}rdr + \int_{R_1}^\infty M_p^q(f_n,r) e^{-q\phi(r)}rdr\\
&=\mathfrak{J}_1 + \mathfrak{J}_2.
\end{align*}
For $\mathfrak{J}_1$, since $\{f_n\}$ converges to $0$ uniformly on compact subsets of $\mathbb{C}$, we have
\begin{equation*}
\mathfrak{J}_1 \leq \sup_{|z|\leq R_1} |f_n(z)|^q \int_0^\infty r e^{-q\phi(r)}dr \lesssim \sup_{|z|\leq R_1} |f_n(z)|^q\to0
\end{equation*}
as $n\to\infty$. For $\mathfrak{J}_2$, since $\{f_n\}$ is a bounded sequence in $F_\varphi^{p,q}$, we get
\begin{equation*}
\mathfrak{J}_2= \int_{R_1}^\infty M_p^q(f_n,r) e^{-q\varphi(r)}r \bigg(e^{q[\varphi(r)-\phi(r)]}\bigg) dr \leq \|f_n\|_{p,q,\varphi}^q \varepsilon^q \lesssim \varepsilon^q.
\end{equation*}
So this gives $\lim\limits_{n\to\infty}\|f_n\|_{p,\phi}=0$. By Lemma \ref{lemma3}, the inclusion $\mathcal{J}$ is compact.

We next show that the operator $\mathcal{H}_\mu^\delta: F_\phi^{p,q}\to F_\varphi^{p,q}$ is bounded. By Minkowski's integral inequality and H\"{o}lder's inequality, we get
\begin{equation*}
M_p^q(\mathcal{H}_\mu^\delta f,r) \lesssim \int_\delta^\infty  \frac{1}{t} M_p^q(f,\frac{r}{t}) d\mu(t).
\end{equation*}
Applying Fubini's theorem, we deduce that
\begin{align}\label{equation5}
\|\mathcal{H}_\mu^\delta f\|_{p,q,\varphi}^q
=\int_0^\infty M_p^q(\mathcal{H}_\mu^\delta f, r)e^{-q\varphi(r)}rdr
\lesssim \int_\delta^\infty  \frac{1}{t} \|\mathcal{D}_tf\|_{p,q,\varphi}^q d\mu(t).
\end{align}
To proceed, we prove that the dilation operator $\mathcal{D}_t: F_\phi^{p,q}\to F_\varphi^{p,q}$ is bounded. If
\begin{equation*}
\lim_{r\to\infty} r\phi'(r)=\infty,
\end{equation*}
then, for arbitrary $M>0$, there exists a $R_2>0$ such that
\begin{equation*}
r\phi'(r)>M,
\end{equation*}
whenever $r>R_2$. In what follows, we take $M=\frac{2}{q}$. For $t\geq \delta$ and $f\in F_\phi^{p,q}$, we have
\begin{align*}
\|\mathcal{D}_t f\|_{p,q,\varphi}^q
&=\int_0^\infty M_p^q(f,\frac{r}{t})e^{-q\varphi(r)}rdr\\
&=\int_0^{tR_2} M_p^q(f,\frac{r}{t})e^{-q\varphi(r)}rdr + \int_{tR_2}^\infty M_p^q(f,\frac{r}{t})e^{-q\varphi(r)}rdr\\
&=\mathfrak{I}_1+\mathfrak{I}_2.
\end{align*}
We estimate $\mathfrak{I}_1$ and $\mathfrak{I}_2$ independently. Firstly, for $\mathfrak{I}_1$, we have
\begin{align*}
\mathfrak{I}_1
&=\int_0^{tR_2} M_p^q(f,\frac{r}{t})e^{-q\varphi(r)}rdr \leq M_p^q(f,R_2) \int_0^{tR_2}e^{-q\varphi(r)}rdr \lesssim M_p^q(f,R).
\end{align*}
On the other hand,
\begin{align*}
\|f\|_{p,q,\phi}^q
&=\int_0^\infty M_p^q(f,r)e^{-q\phi(r)}rdr \geq \int_{R_2}^\infty M_p^q(f,r) e^{-q\phi(r)}rdr\\
&\geq M_p^q(f,R_2) \int_{R_2}^\infty e^{-q\phi(r)}rdr\\
&\simeq M_p^q(f,R_2).
\end{align*}
So we obtain the inequality $\mathfrak{I}_1 \lesssim \|f\|_{p,q,\phi}^q$, where the constants are independent of $t$.

For $\mathfrak{I}_2$, we have
\begin{align*}
\mathfrak{I}_2
&=\int_{tR_2}^\infty M_p^q(f,\frac{r}{t})e^{-q\varphi(r)}rdr=\int_{R_2}^\infty M_p^q(f,s) t^2e^{-q\varphi(ts)} sds\\
&\leq \int_{R_2}^\infty M_p^q(f,s) \bigg(\sup_{t\geq \delta}t^2e^{-q\varphi(ts)}\bigg) sds.
\end{align*}
We check
\begin{equation*}
\sup_{t\geq \delta} t^2e^{-q\varphi(ts)} = \delta^2 e^{-q\phi(s)}.
\end{equation*}
Let $s>R_2$. Set
\begin{equation*}
h(t)=t^2e^{-q\varphi(ts)}, \quad t\geq\delta.
\end{equation*}
We claim that that the function $h(t)$ is decreasing on $[\delta,\infty)$. In fact, since
\begin{align*}
h'(t)&=2te^{-q\varphi(ts)} + t^2e^{-q\varphi(ts)} [-q\varphi'(ts)s]\\
     &=2te^{-q\varphi(ts)}\bigg[1-\frac{q}{2}\varphi'(ts)ts\bigg],
\end{align*}
\begin{equation*}
1-\frac{q}{2}\varphi'(ts)ts=1-\frac{q}{2}\phi'\bigg(\frac{t}{\delta}s\bigg)\frac{t}{\delta}s, \quad \frac{t}{\delta}\geq 1,
\end{equation*}
and
\begin{equation*}
\phi'\bigg(\frac{t}{\delta}s\bigg)\frac{t}{\delta}s>\frac{2}{q},
\end{equation*}
we have $h'(t)<0$ for all $t\geq\delta$. So we obtain $\mathfrak{I}_2 \lesssim \|f\|_{p,q,\phi}^q$, where the constants are independent of $t$ again. From this, we get that $\mathcal{D}_t: F_\phi^{p,q}\to F_\varphi^{p,q}$ is bounded. By (\ref{equation5}) and $\int_{(0,\infty)}t^{-1}d\mu(t)<\infty$, we see that $\mathcal{H}_\mu^\delta:F_\varphi^p\to F_\phi^p$ is bounded. This, together with the compactness of the inclusion $\mathcal{J}: F_\varphi^{p,q}\to F_\phi^{p,q}$ yields that $\mathcal{H}_\mu^\delta$ is compact on $F_\phi^{p,q}$.

Finally, for any $f\in F_\phi^{p,q}$, by Proposition \ref{proposition}, we get
\begin{equation*}
\|\mathcal{H}_\mu f - \mathcal{H}_\mu^\delta f\|_{p,q,\phi}^q \lesssim \int_{[1,\delta]}  \frac{1}{t} \|\mathcal{D}_tf\|_{p,q,\phi}^q d\mu(t)\leq \mu([1,\delta])\|f\|_{p,q,\phi}^q.
\end{equation*}
Therefore,
\begin{equation*}
\|\mathcal{H}_\mu - \mathcal{H}_\mu^\delta\|^q \lesssim \mu([1,\delta]).
\end{equation*}
Let $\delta\to1^+$, by $\mu(\{1\})=0$, we obtain that $\mathcal{H}_\mu$ is compact on $F_\phi^{p,q}$.

Conversely, suppose that $\mathcal{H}_\mu$ is compact on $F_\phi^{p,q}$. Then $\mathcal{H}_\mu$ is bounded and $\mu(0,1)=0$ and $\int_{[1,\infty)}t^{-1}d\mu(t)<\infty$. Consider the normalized functions
\begin{equation*}
\widetilde{h}_n(z)=\frac{z^n}{\|z^n\|_{p,q,\phi}},  \quad  n\in\mathbb{N}_0.
\end{equation*}
Then the sequence $\{\widetilde{h}_n\}$ is bounded in $F_\phi^{p,q}$ and tends to zero uniformly on compact subsets of $\mathbb{C}$. Moreover, for each $n\in\mathbb{N}_0$,
\begin{equation*}
\mathcal{H}_\mu \widetilde{h}_n = \int_{[1,\infty)}\frac{1}{t^{n+1}}d\mu(t) \widetilde{h}_n.
\end{equation*}
It follows from the compactness of $\mathcal{H}_\mu$ and Lemma \ref{lemma3} that
\begin{equation*}
\lim_{n\to\infty}\int_{[1,\infty)}\frac{1}{t^{n+1}}d\mu(t)=\lim_{n\to\infty}\|\mathcal{H}_\mu h_n\|_{p,q,\phi}=0.
\end{equation*}
By Lemma \ref{lemma4}, we get that $\mu(\{1\})=0$ and so $\mu((0, 1])=0$. The proof is complete.
\end{proof}


\section{Power boundedness and uniformly mean ergodicity of $\mathcal{H}_\mu$}\label{section4}

Another interesting notion closely related to boundedness of an operator is power boundedness. This notion finds many applications in the study of the ergodicity, stability, and dynamics of linear operator (see \cite{Bonet1, JJKS, KMOT, SM}). In this section, we investigate when Hausdorff operator is power bounded or uniformly mean ergodic on weighted mixed norm Fock spaces.


Let $X$ be a Banach space and $T$ a bounded linear operator on $X$. Then, for $n\in\mathbb{N}$, we set the operator $T^n$ as the $n$th iterate of $T$. Recall that $T$ is power bounded if
\begin{equation*}
\sup_{n\in\mathbb{N}_0}\|T^n\|<\infty.
\end{equation*}
An operator $T$ is said to be quasi-compact if $T^m$ is compact for some $m\in\mathbb{N}$, equivalently, there exists an integer $m$ and a compact
operator $Q$ such that $\|T^m-Q\|<1$. Quasi-compact operators share some properties of compact operators, in particular their spectrum reduces to eigenvalues and $\{0\}$. The Ces\`{a}ro means of $T$ are defined by
\begin{equation*}
T_{[n]}=\frac{1}{n}\sum_{k=1}^n T^k, \quad n\in\mathbb{N}_0.
\end{equation*}
The following equality is well-known and can be checked easily
\begin{equation}\label{equation6}
\frac{1}{n}T^n=T_{[n]}-\frac{n-1}{n}T_{[n-1]}, \quad n\in\mathbb{N}_0,
\end{equation}
where $T_{[0]}=I$ is the identity operator on $X$. A bounded linear operator $T$ on a Banach space $X$ is said to be mean ergodic if there exists a bounded linear operator $P$ on $X$ such that
\begin{equation*}
Px=\lim_{n\to\infty}T_{[n]} x,  \quad  x\in X
\end{equation*}
exists in $X$. If the convergence is in the operator norm, then $T$ is called uniformly mean ergodic. Clearly, by (\ref{equation6}), if $T$ is uniformly mean ergodic, then
\begin{equation}\label{equation4}
\lim_{n\to\infty}\frac{\|T^n\|}{n}=0.
\end{equation}
See \cite{Krengel} for ergodic theory of operators on Banach spaces.

We now state the main result for power boundedness of $\mathcal{H}_\mu$.
\begin{theorem}\label{power bounded}
Let $\mu$ be a positive Borel measure on $(0,\infty)$ and let $1\leq p,q\leq\infty$, $\phi$ be a weight. Then the following conditions are equivalent.

$\mathrm{(i)}$ $\mathcal{H}_\mu$ is power bounded on $F_\phi^{p,q}$.

$\mathrm{(ii)}$ $\mathcal{H}_\mu$ is a contraction on $F_\phi^{p,q}$.

$\mathrm{(iii)}$ $\mu(0,1)=0$ and $\int_{[1,\infty)}\frac{1}{t}d\mu(t)\leq1$. Moreover, for each $n\in\mathbb{N}$, we have
\begin{equation}\label{equation3}
\|\mathcal{H}_\mu^n\|=\|\mathcal{H}_\mu\|^n=\bigg(\int_{[1,\infty)}\frac{1}{t}d\mu(t)\bigg)^n.
\end{equation}
\end{theorem}
\begin{proof}
It is obvious that $\mathrm{(ii)}$ implies $\mathrm{(i)}$. The equivalence of $\mathrm{(ii)}$ and $\mathrm{(iii)}$ follows from Theorem \ref{boundedness}. To complete the proof, it remains to show that condition $\mathrm{(i)}$ implies $\mathrm{(ii)}$. Let us assume that $\mathcal{H}_\mu$ is power bounded. Since $\mu(0,1)=0$, we have
\begin{equation*}
\mathcal{H}_\mu f(z)=\int_{[1,\infty)} \frac{1}{t} f \bigg(\frac{z}{t}\bigg) d\mu(t).
\end{equation*}
By $n$ applications of Minkowski's integral inequality, $M_p(f,r)$ is increasing with $r\in[0,\infty)$, we get
\begin{align*}
M_p(\mathcal{H}_\mu^n f,r)
&\leq \int_{[1,\infty)}  \frac{1}{t} M_p\bigg(\mathcal{H}_\mu^{n-1} f,\frac{r}{t}\bigg) d\mu(t)
 \leq \int_{[1,\infty)}\frac{1}{t}d\mu(t) M_p(\mathcal{H}_\mu^{n-1} f,r)\\
&\leq \cdots \leq \bigg(\int_{[1,\infty)}\frac{1}{t}d\mu(t)\bigg)^n M_p(f,r).
\end{align*}
Multiplying both sides by $re^{-q\phi(r)}$ and integrating over $r$ yields
\begin{equation*}
\|\mathcal{H}_\mu^n f\|_{p,q,\phi} \leq \bigg(\int_{[1,\infty)}\frac{1}{t}d\mu(t)\bigg)^n \|f\|_{p,q,\phi},
\end{equation*}
which implies that $\mathcal{H}_\mu^n$ is bounded on $F_\phi^{p,q}$ and
\begin{equation*}
\|\mathcal{H}_\mu^n\| \leq \bigg(\int_{[1,\infty)}\frac{1}{t}d\mu(t)\bigg)^n.
\end{equation*}
On the other hand, if $\mathcal{H}_\mu^n$ is bounded, then
\begin{equation*}
\|\mathcal{H}_\mu^n 1\|_{p,q,\phi} = \bigg(\int_1^\infty\frac{1}{t}d\mu(t)\bigg)^n \|1\|_{p,q,\phi}\leq \|\mathcal{H}_\mu^n\|\|1\|_{p,q,\phi},
\end{equation*}
and hence
\begin{equation*}
\|\mathcal{H}_\mu^n\|\geq \bigg(\int_{[1,\infty)}\frac{1}{t}d\mu(t)\bigg)^n.
\end{equation*}
So we get
\begin{equation*}
\|\mathcal{H}_\mu^n\|= \bigg(\int_{[1,\infty)}\frac{1}{t}d\mu(t)\bigg)^n.
\end{equation*}
This, together with (\ref{equation2}) yields that
\begin{equation*}
\|\mathcal{H}_\mu^n\|=\|\mathcal{H}_\mu\|^n.
\end{equation*}
We obtain equality (\ref{equation3}) of the norm and the integral. Therefore, $\|\mathcal{H}_\mu\|\leq1$ and condition $\mathrm{(ii)}$ holds. The proof is complete.
\end{proof}

Having identified conditions under which $\mathcal{H}_\mu$ is power bounded, we turn our attention to the uniformly mean ergodicity of $\mathcal{H}_\mu$ on $F_\phi^{p,q}$. We first state the following result about compact Hausdorff operators on weighted mixed norm Fock spaces.
\begin{theorem}\label{compact UME}
Let $\mu$ be a positive Borel measure on $(0,\infty)$ satisfying $\mu((0,1])=0$, $\phi\in\mathcal{C}$, and $1\leq p,q\leq\infty$. Then the Hausdorff operator $\mathcal{H}_\mu$ is uniformly mean ergodic on $F_\phi^{p,q}$ if and only if
\begin{equation*}
\int_{(1,\infty)}\frac{1}{t}d\mu(t)\leq1.
\end{equation*}
\end{theorem}
\begin{proof}
Assume that $\mathcal{H}_\mu$ is uniformly mean ergodic. By (\ref{equation4}) and (\ref{equation3}), we have
\begin{equation*}
\frac{1}{n}\|\mathcal{H}_\mu\|^n=\frac{1}{n}\|\mathcal{H}_\mu^n\|\to0
\end{equation*}
as $n\to\infty$. It follows that $\|\mathcal{H}_\mu\|\leq1$. By (\ref{equation2}), we get
\begin{equation*}
\int_{(1,\infty)}\frac{1}{t}d\mu(t)\leq1
\end{equation*}

Conversely, if
\begin{equation*}
\int_{(1,\infty)}\frac{1}{t}d\mu(t)\leq1,
\end{equation*}
then $\mathcal{H}_\mu$ is a contraction and so it is power bounded. By Theorem \ref{compactness}, we see that $\mathcal{H}_\mu$ is compact on $F_\phi^{p,q}$. This, together with the power boundedness of $\mathcal{H}_\mu$, gives that $\mathcal{H}_\mu$ is uniformly mean ergodic on $F_\phi^{p,q}$ (see Yosida and Kakutani's Theorem \cite{YK}: If an operator $T$ is power bounded and quasi-compact on $X$, then $T$ is uniformly mean ergodic). The proof is complete.
\end{proof}

An interesting consequence about the compact Hausdorff operator on weighted mixed norm Fork spaces is the following.
\begin{corollary}
Let $1\leq p,q\leq\infty$ and $\mathcal{H}_\mu$ be a compact operator on $F_\phi^{p,q}$. Then the following conditions are equivalent.

$\mathrm{(i)}$ $\mathcal{H}_\mu$ is power bounded.

$\mathrm{(ii)}$ $\mathcal{H}_\mu$ is a contraction.

$\mathrm{(iii)}$ $\mathcal{H}_\mu$ is uniformly mean ergodic.
\end{corollary}

We next consider the uniformly mean ergodicity for bounded Hausdorff operator $\mathcal{H}_\mu$ on $F_\phi^{p,q}$.
\begin{theorem}\label{bounded UME}
Let $\mu$ be a positive Borel measure on $(0,\infty)$ satisfying the condition $\mu(0,1)=0$, $1\leq p,q\leq\infty$, and $\phi$ be a weight.

$\mathrm{(i)}$ If $\mathcal{H}_\mu$ is uniformly mean ergodic on $F_\phi^{p,q}$, then
\begin{equation*}
\int_{[1,\infty)}\frac{1}{t}d\mu(t)\leq1.
\end{equation*}

$\mathrm{(ii)}$ If
\begin{equation*}
\int_{[1,\infty)}\frac{1}{t}d\mu(t)<1,
\end{equation*}
then $\mathcal{H}_\mu$ is uniformly mean ergodic on $F_\phi^{p,q}$, and $\lim\limits_{n\to\infty}{{\mathcal{H}_{\mu}}_{[n]}}=0$.
\end{theorem}
\begin{proof}
The proof of part $\mathrm{(i)}$ is similar to the necessity in Theorem \ref{compact UME}. We prove $\mathrm{(ii)}$. In fact, if $\int_{[1,\infty)}t^{-1}d\mu(t)<1$, then, by (\ref{equation3}), we arrive at the following estimate
\begin{equation*}
\bigg\| \frac{1}{n}\sum_{k=1}^n \mathcal{H}_\mu^k \bigg\| \leq \frac{1}{n}\sum_{k=1}^n \| \mathcal{H}_\mu^k \|
=\frac{1}{n}\sum_{k=1}^n \| \mathcal{H}_\mu \|^k=\frac{1}{n}\frac{\| \mathcal{H}_\mu \|(1-\| \mathcal{H}_\mu \|^n)}{1-\| \mathcal{H}_\mu \|}\to0
\end{equation*}
as $n\to\infty$. Hence $\mathcal{H}_\mu$ is uniformly mean ergodic. The proof is complete.
\end{proof}

\section{Further comments}\label{section5}

In this section, we discuss some problems closely related to the results presented in this paper, and pose a conjecture. In Section \ref{section4}, we studied the power bounded and uniformly mean ergodic condition for Hausdorff operator. The characterization of dynamics of such operator is also an important related topic, which remains open (see \cite{BM} for the theory of linear operator dynamics). It would be interesting to explore the dynamical properties of Hausdorff operator.

In Theorem \ref{bounded UME}, we show that for $\int_{[1,\infty)} t^{-1} d\mu(t)<1$, Hausdorff operators on weighted mixed norm Fock spaces are uniformly mean ergodic. On the other hand, condition $\int_{[1,\infty)} t^{-1} d\mu(t)=1$ is necessary for the uniformly mean ergodicity of Hausdorff operators. It is thus natural to ask, if
\begin{equation*}
\int_{[1,\infty)} \frac{1}{t} d\mu(t)=1,
\end{equation*}
is Hausdorff operator $\mathcal{H}_\mu$ uniformly mean ergodic?  By \cite[Theorem 2.1]{Krengel}, this question is equivalent to whether the range of $I-\mathcal{H}_\mu$, denoted as $\mathrm{Im}(I-\mathcal{H}_\mu)$, is closed in $F_\phi^{p,q}$, where
\begin{equation*}
(I-\mathcal{H}_\mu)f(z)=\int_{[1,\infty)}\frac{1}{t} \bigg[f(z)-f\bigg(\frac{z}{t}\bigg)\bigg] d\mu(t).
\end{equation*}
We know that for Hardy operator $H$, $\text{Im}(I-H)$ is closed \cite[Corollary 4.2]{Beltran}. On the other hand, note that $\text{Ker}(I-\mathcal{H}_\mu)$ contains only the constant functions. In fact, for any non constant function $f(z)=\sum_{n=0}^\infty a_nz^n\in F_\phi^{p,q}$, there exists a certain term $a_n\neq0$ such that $(I-\mathcal{H}_\mu)f\neq0$. Therefore, it is natural to pose the following conjecture.

\begin{conjecture}
Let $\mu$ be a positive Borel measure on $(0,\infty)$ satisfying $\mu(0,1)=0$, $1\leq p,q\leq\infty$, and $\phi$ be a weight. If $\int_{[1,\infty)} t^{-1} d\mu(t)=1$, then $\mathrm{Im}(I-\mathcal{H}_\mu)$ is the closed subspace of $F_\phi^{p,q}$ given by
\begin{equation*}
\mathrm{Im}(I-\mathcal{H}_\mu)=\{f\in F_\phi^{p,q}: f(0)=0\}.
\end{equation*}
\end{conjecture}

\end{document}